\documentclass[12pt,oneside,a4paper]{article}
\usepackage[utf8]{inputenc}
\usepackage[T2A,T1]{fontenc}
\usepackage[english]{babel}
\usepackage{amssymb}
\usepackage{amsthm}
\usepackage{amsmath}
\usepackage{hyperref}
\hypersetup{
    colorlinks,
    linkcolor=red,
    citecolor=green,
    urlcolor=blue
}
\usepackage{verbatim}
\usepackage{enumitem}
\usepackage[dvips]{graphicx}
\usepackage{geometry}
\usepackage[numbers,sort&compress]{natbib}
\usepackage{tikz}
\usepackage{tkz-graph}
\usepackage{caption}
\newtheorem{thm}{Theorem}
\newtheorem{lem}[thm]{Lemma}

\newtheorem{prop}[thm]{Proposition}

\newtheorem{defn}[thm]{Definition}

\newtheorem{question}[thm]{Question}
\newtheorem{conj}[thm]{Conjecture}

\newtheorem{ex}[thm]{Example}

\newtheorem{rem}[thm]{Remark}

\renewcommand{\leq}{\leqslant}
\renewcommand{\geq}{\geqslant}
\renewcommand{\le}{\leqslant}
\renewcommand{\ge}{\geqslant}

\begin{document}

\author{Ivailo Hartarsky\thanks{D\'epartement de Math\'ematiques et Applications, \'Ecole Normale Sup\'erieure, CNRS, PSL Research University, Sorbonne Univ\'ersit\'es, 45 rue d'Ulm, Paris, France.\newline E-mail address: \textsf{ivailo.hartarsky@ens.fr}}}

\title{Maximal Bootstrap Percolation Time on the Hypercube via Generalised Snake-in-the-Box}

\pagestyle{plain}

\maketitle
\begin{abstract}
In $r$-neighbour bootstrap percolation, vertices (sites) of a graph $G$ become ``infected'' in each round of the process if they have $r$ neighbours already infected. Once infected, they remain such. An initial set of infected sites is said to percolate if every site is eventually infected. We determine the maximal percolation time for $r$-neighbour bootstrap percolation on the hypercube for all $r \geq 3$ as the dimension $d$ goes to infinity up to a logarithmic factor. Surprisingly, it turns out to be $\frac{2^d}{d}$, which is in great contrast with the value for $r=2$, which is quadratic in $d$, as established by Przykucki~\cite{Przykucki12}. Furthermore, we discover a link between this problem and a generalisation of the well-known Snake-in-the-Box problem.
\end{abstract}

\section{Introduction}
Bootstrap percolation was introduced in 1979 by Chalupa, Leath and Reich~\cite{Chalupa79} as a simplified monotone version of ferromagnetic dynamics and it is in particular related to Glauber dynamics of the Ising model. The general $r$-neighbour model on a graph $G$ is defined as follows. Consider an initial subset of the vertices (sites) that are declared infected. At each time step every site becomes infected if it has at least $r$ neighbours already infected and infected site always remain such. We say that percolation occurs if eventually all sites of $G$ are infected. In the most classical setting the initially infected sites are selected randomly and independently with probability $p$ and the graph $G$ is taken to be a finite $d$-dimensional grid $\{1,\ldots, n\}^d$.

One of the founding results in the field was by Aizenman and Liebowitz~\cite{Aizenman88}, who determined the order of the critical probability of percolation for $r=2$, all fixed $d$ as $n\rightarrow\infty$. The simplest setting, $r=d=2$ was then studied by Holroyd~\cite{Holroyd03}, who proved that the threshold is sharp and determined the leading term of the critical probability. Further work on that threshold was done and the order of the second term is now known~\cite{Gravner08,Hartarsky18Rob}. However, the case $r>2$ required a lot more care, because the stable sets of infected sites are no longer simple boxes. An important step was done by Cerf and Cirillo~\cite{Cerf99} and Cerf and Manzo~\cite{Cerf02}, who proved the counterpart of the result of~\cite{Aizenman88}. Their methods were later used in conjunction with Holroyd's to determine the leading term of the critical probability for all fixed $r$ and $d$ when $n\rightarrow\infty$~\cite{Balogh12}.

A less standard and more combinatorial facet of bootstrap percolation consists in keeping $n$ fixed and letting $d$ grow to infinity, so that the simplest case is the high dimensional hypercube. This setting was explored by Balogh and Bollob\'as~\cite{Balogh06}, and later Balogh, Bollob\'as, and Morris~\cite{Balogh10} determined the critical probability of percolation with high precision for $r=2$, and also for high dimensional grids with size not necessarily equal to $2$. However, the situation for $r>2$ remains entirely open due to the lack of tools to handle the more complicated stable sets, since the method of~\cite{Cerf99,Cerf02} is no longer of relevance.

Alongside the probabilistic perspective on bootstrap percolation, purely combinatorial extremal questions have been widely investigated. Such deterministic bounds have proved useful for obtaining probabilistic results as well, e.g. in~\cite{Balogh10}. However, it has become customary to expect the unexpected, as, more often than not, answers to such extremal questions are very counterintuitive and very far from ``common'' behaviour. Some deal with the classical $2$ dimensional $2$-neighbour model, like~\cite{Morris09,Benevides13,Benevides15}, but others~\cite{Riedl10,Morrison17,Morrison18} focus on the hypercube. The typical quantities assessed are the extremal sizes of (extremal) (non-)percolating sets, extremal percolation time or mixtures of those.

One such result by Przykucki~\cite{Przykucki12} concerns the maximal percolation time on the hypercube under the $2$-neighbour model. Contrary to the result of~\cite{Benevides15} that the maximal percolation time for the same model in two dimensions is of the order of the size of the whole grid considered, for the hypercube of dimension $d$ the maximal percolation time was determined to be merely $\left\lfloor\frac{d^2}{3}\right\rfloor$. Based on the construction in~\cite{Przykucki12} one might expect that setting $r=3$ would simply allow one to gain another factor of order $d$ and the percolation time to be at most cubic in the dimension. Most surprisingly, we prove that there is a drastic jump between $r=2$ and $r=3$ for this question. We show that the maximal percolation time goes from close to the trivial lower bound $1$, as found in~\cite{Przykucki12}, to close to the trivial upper bound $2^d$. More precisely, we prove that for all $r>2$ the maximal percolation time is equal to $\frac{2^d}{d}$ up to a logarithmic factor.

The lower bound is based on an entirely explicit construction, though quite elaborate, as it also uses a previously known non-trivial one. An essential ingredient for this bound is a new link we establish between bootstrap percolation and the very well-known snake-in-the-box problem, which concerns long induced paths and cycles in the hypercube. It was introduced by Kautz in the late 50s~\cite{Kautz58} and has a wide range of applications, namely in coding, error-correction and others. It was first proved in~\cite{Glagolev70,Evdokimov69} that the maximal length of a snake-in-the-box is $2^d$ up to a constant factor, though its correct asymptotic value is not yet known.

We will rather be concerned with a natural generalisation of the problem, introduced by Singleton in~\cite{Singleton66}. It asks for a long path (or cycle) in the hypercube such that sites at distance at least $k$ along the path are also at distance at least $k$ in the hypercube as well. Hence, the snake-in-the-box problem corresponds to $k=2$. These paths or cycles are usually referred to as snakes or circuit codes of spread $k$, but we will call them $k$-snakes and we will only need $3$-snakes for our result. The maximal length of $k$-snakes was also studied extensively over the last half a century. A very easy upper bound for $k=3$, mentioned already in~\cite{Singleton66} is $\frac{2^d}{d-2}$, is fairly close to the right asymptotics. The right exponent $2^{d-o(d)}$ was determined in~\cite{Kurljandcik71} and Evdokimov determined the maximal length of a $3$-snake to be $\frac{2^d}{d}$ up to a logarithmic factor~\cite{Evdokimov76}. This result is at the base of our construction. For a more recent overview, which is very complete from the mathematical perspective, on snake-in-the-box and related problems, the reader is referred to the survey~\cite{Evdokimov14} by the same author.

We should also note that, curiously, another link between bootstrap percolation and the snake-in-the-box (with spread $k=2$) problem has been observed in~\cite{Shende15}, although it is along an entirely different direction and very specific to $r=2$.

\section{Notation}
In this section, we introduce the notation necessary for the proof of the main result.

We denote by $M_r(d)$ our quantity of interest -- the maximal time of $r$-neighbour bootstrap percolation on the $d$-dimensional hypercube $\{0,1\}^d$ with its usual graph structure. Denote by $d(\cdot,\cdot)$ the associated graph distance induced by the norm $\|x\|=\sum_{i=1}^dx_i$ for $x=(x_i)\in\{0,1\}^d$.

\paragraph{Snakes}
\begin{defn} For $k\ge 1$ a \emph{$k$-snake} is a path $(S_t)_{t=0}^T$ in the hypercube such that, for all $t\ge 0$ and $t'\in[t+k,T]$ it holds that $d(S_t,S_{t'})\ge k$. We call $T$ the \emph{length} of $S$ and refer to the parameter $t$ as the \emph{time}.
\end{defn}
\begin{rem}
For a $k$-snake of length greater than $k$ this definition implies that for $t,t'\in[0,T]$ such that $|t-t'|\leq k$ one has $d(S_t,S_{t'})=|t-t'|$. Indeed, each step increases the distance by at most $1$, but after $k$ steps we are required to be at distance at least $k$, so that increasing by $1$ was always necessary. Hence, snakes are $k$-locally isometric to paths.
\end{rem}
\begin{defn}
We denote by $s(d)$ the maximal length of a $3$-snake in the $d$-dimensional hypercube.
\end{defn}

The following bound was established by Evdokimov~\cite{Evdokimov76}.
\begin{prop}
\label{prop:Evdokimov}
For all $d\geq 3$
\[s(d)\geq \frac{2^d}{d(\log d)^2}\,.\]
\end{prop}
	
\paragraph{Hypercubes}
For the remainder of this paper, we employ the convenient notation used by Przykucki in~\cite{Przykucki12}. Though it may appear very technical at first, it will prove itself to be very practical.
\begin{defn}
For any finite sequence $(a_i)\in \{0,1,*\}^n$ we denote 
\[[a_1,\ldots, a_n]:=\{(b_1,\ldots,b_n)\in\{0,1\}^n\mid \forall\, 1\leq i\leq n,\, a_i\neq *\Rightarrow b_i=a_i\}\]
and call all such sets \emph{subcubes} (of the hypercube $\{0,1\}^n$). We extend this notation to the concatenation of two sequences $(a_i)$ and $(b_i)$ as
\[[a_1,\ldots, a_n][b_1,\ldots, b_k]:=[a_1,\ldots, a_n,b_1,\ldots, b_k]\,.\]
$[a_1,\ldots, a_n]^k$ stands for $[a_1,\ldots,a_n]\ldots[a_1,\ldots,a_n]$ repeating $k$ times. We will abusively identify singletons with their unique element, when they arise in this notation, i.e. when $*$ is never used.
\end{defn}
\begin{ex}
The hypercube of dimension $d$ is thus denoted by $[*]^d$ and $[1,0,1][*]^{d-6}[0]^3$ is its $d-6$ dimensional subcube whose first three coordinates are $1$, $0$, and $1$ in that order, and whose last three coordinates are all $0$.

We may write $[1,0,1][0]^2$ for both the site $(1,0,1,0,0)$ and the subcube $\{(1,0,1,0,0)\}$.
\end{ex}

\begin{defn}
For any sequence of sequences $\left((a^j_i)_{i=1}^{l_j}\right)_{j=1}^n$ on the alphabet $\{0,1,*\}$, we define their \emph{permutation}
\[\overline{[a^1_1,\ldots,a^1_{l_1}]\ldots[a^n_1,\ldots,a^n_{l_n}]}=\bigcup_{\sigma\in\mathfrak{S}_n}\left[a^{\sigma(1)}_1\ldots,a^{\sigma(1)}_{l_{\sigma(1)}}\right]\ldots\left[a^{\sigma(n)}_1\ldots,a^{\sigma(n)}_{l_{\sigma(n)}}\right]\,\]
where $\mathfrak{S}_{n}$ is the symmetric group.
\end{defn}
\begin{ex}The elements of $[0]\overline{[0]^2[1,0]}[*]$ are $(0,0,0,1,0,0)$, $(0,0,0,1,0,1)$, $(0,0,$ $1,0,0,0)$, $(0,0,1,0,0,1)$, $(0,1,0,0,0,0)$, $(0,1,0,0,0,1)$.
\end{ex}
It is important to note that the permutation does not act inside each component of the concatenation and moves the whole blocks without interlacing them, so that $\overline{[1]^2[0]^2}$, $\overline{[1,1][0,0]}$ and $\overline{[1,0]^2}$ are all different sets with $6$, $2$ and $1$ elements respectively.

\section{The Main Result}
\label{sec:main}
In this section, we prove our main result determining the maximal percolation time for all $r>2$ in the hypercube up to a $\log d$ factor.
\begin{thm}
\label{th:main}
For all $r\geq3$
\[M_r(d)=\frac{2^d}{d} (\log d)^{-O(1)}\,.\]
\end{thm}
The upper and lower bounds are established independently. We start with the lower one, which will follow by linking the bootstrap process to long $3$-snakes.
\begin{lem}
\label{lem:time}
Let $d\geq 15$ be odd. Then,
\[s(d-10)\leq M_3(d)\]
\end{lem}
Let us sketch the idea before we turn to the proof of the lemma. We would like have a long $3$-snake becoming infected one site at a time. To achieve that we fix a long $3$-snake in a subcube of codimension $9$ and infect neighbours of that snake in new directions in order to have two for each site of the snake. Then we only need to have the beginning of the snake initially infected. We also make sure that next to the end of the snake there is a configuration of lots of infected sites which can percolate only using the end of the snake in addition. Of course, some care is needed in order not to infect any other site by accident before the snake can reach its end.

\begin{proof}[Proof of Lemma~\ref{lem:time}]
We will need the following technical lemma.
\begin{lem}
\label{cor:snakemodified}
Let $d''\geq 6$. There is a $3$-snake $S$ of dimension $d''$ and length $T$ such that the following conditions all hold.
\begin{enumerate}
\item $S_{T-3}=[1,0,1,0,1][0]^{d''-5}$.
\item $S_{T-2}=[1,0,1][0]^{d''-3}$.
\item $S_{T-1}=[1][0]^{d''-1}$.
\item \label{cond:cor1}$S_T=[0]^{d''}$.
\item \label{cond:cor3}$\|S_t\|>3$ for every $t<T-3$.
\item \label{cond:cor2}$T\geq s(d''-1)$.
\end{enumerate}
\end{lem}
\begin{proof}[Proof of Lemma~\ref{cor:snakemodified}]
It suffices to satisfy conditions~\ref{cond:cor1}-\ref{cond:cor2} and then to permute the coordinates to also fulfil the other conditions. One can achieve conditions~\ref{cond:cor1}-\ref{cond:cor2} as follows.

Let $S'$ be a $d''-1$-dimensional $3$-snake of maximal length with $S'_{s(d''-1)}=[0]^{d''-1}$ (to obtain it compose a $d''-1$-dimensional $3$-snake of maximal length by a suitable isomorphism of the hypercube). Then we set $S_t=[1]S'_t$ for all $0\leq t\leq s(d''-1)$, $T=s(d''-1)+1$ and $S_{T}=[0]^{d''}$. Conditions~\ref{cond:cor1} and~\ref{cond:cor2} are clearly satisfied. Furthermore, since $S'$ is a $3$-snake ending in $[0]^{d''-1}$, we have that $\|S'_t\|\geq 3$ for all 
$t\leq s(d''-1)-3$. Thus, condition~\ref{cond:cor3} does hold by construction and $S$ is indeed a $3$-snake.
\end{proof}

For convenience denote $d':=d-3$ and $d'':=d-9$. Let $S$ be as provided by Lemma~\ref{cor:snakemodified}. Let the initial set of infected sites $I$ be defined as follows (see Figure~\ref{fig:cube}).
\begin{itemize}
\item Infect $[0]^9S_0$.
\item For $i\in\{1,2,3\}$ set
\[S^i:=\left\{S_{T-i-3j},0\leq j\leq \frac{T-i}{3}\right\}\,.\]
Infect
\[I_0:=[0]^3\overline{[0][1]}[0]^4S^1\cup[0]^5\overline{[0][1]}[0]^2S^2\cup[0]^7\overline{[0][1]}S^3\,.\]
Do note that we do not include neighbours of the end of the snake $[0]^9S_T$.
\item Moreover, infect
\[J_1:=[1,1][*]^{d'+1},\; J_2:=\overline{[0][1]}[1][0]^{d'}\;\mathrm{and}\;J_3:=\overline{[0][1]}[0]\overline{[1,1][0,0]^{\frac{d'-2}{2}}}\,.\]
Recall that $\overline{[1,1][0,0]^k}=\bigcup_{0\le l\leq k}[0,0]^l[1,1][0,0]^{k-l}$.
\end{itemize}

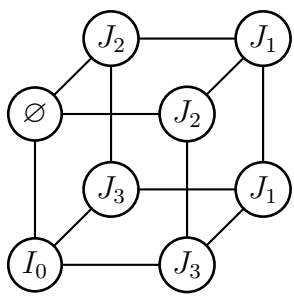
\begin{figure}
\begin{center}
\begin{tikzpicture}
  \SetGraphUnit{3}
\GraphInit[vstyle=Normal]
  \SetVertexNormal[Shape      = circle,
                   LineWidth  = 1pt,
                   MinSize = 20pt]
  \SetVertexMath
  \Vertex[x=0,y=0,L=I_0]{A}
  \Vertex[x=2,y=0,L=J_3]{B}
  \Vertex[x=2,y=2,L=J_2]{C}
  \Vertex[x=0,y=2,L=\varnothing]{D}
  \Vertex[x=1,y=1,L=J_3]{E}
  \Vertex[x=3,y=1,L=J_1]{F}
  \Vertex[x=3,y=3,L=J_1]{G}
  \Vertex[x=1,y=3,L=J_2]{H}
  \Edge(A)(B)
  \Edge(A)(D)
  \Edge(A)(E)
  \Edge(B)(C)
  \Edge(B)(F)
  \Edge(C)(D)
  \Edge(C)(G)
  \Edge(D)(H)
  \Edge(E)(F)
  \Edge(E)(H)
  \Edge(F)(G)
  \Edge(G)(H)
\end{tikzpicture}
\end{center}
\caption{Each vertex of the cube in this picture represents a $d'$ dimensional hypercube, so that only the first three dimensions of $[*]^d$ are visible. We indicate the positions of the different parts of the initial infected set $I$. $J_1$ consists of two entire $d'$-dimensional subcubes, $J_2$ has one site in each of the two subcubes indicated, $J_3$ has $\frac{d'}{2}$ sites at distance $4$ in each of the two subcubes indicated. Finally, $I_0$ contains two neighbours of each site in the $3$-snake $[0]^{9}S$ (except its end). The $3$-snake in question lies in the same $d'$ dimensional subcube as $I_0$.}
\label{fig:cube}
\end{figure}

We claim that $[0]^9S$ is infected one site at a time, that no site outside $[0]^9S$ is infected strictly before $S_{T}$ and that percolation occurs. However, before we turn to the proof of those claims, let us establish some properties of the configuration.
\begin{itemize}
\item $I_0$, $J_2$ and $J_3$ have pairwise no common neighbours. Indeed,
\begin{itemize}
\item $J_2$ and $J_3$ have no common neighbours by parity.
\item $J_2$ and $I_0$ have no common neighbours, since $[0]^d\not\in I_0$ and two of the first three coordinates are different.
\item $J_3$ and $I_0$ have no common neighbours. To see this, consider a site $j$ in $J_3$ at distance $2$ from $i\in I_0$. Those two differ in one of the first three coordinates, so $i$ has a neighbour in $[0]^3\overline{[1,1][0,0]^{\frac{d'-2}{2}}}$. Then $\|i\|\in\{1,3\}$ and by condition~\ref{cond:cor3} of Lemma~\ref{cor:snakemodified} and parity $i$ is necessarily a neighbour of $[0]^9S_{T-2}$ (recall that $I_0$ does not contain neighbours of $[0]^9S_T$). Hence, $i\in[0]^3\overline{[0]^5[1]}[1,0,1][0]^{d''-3}$. Notice that $i$ necessarily has $2$ adjacent $1$s, since it has a neighbour in $[0]^3\overline{[1,1][0,0]^{\frac{d'-2}{2}}}$. However, this is the case only if $i=[0]^8[1,1,0,1][0]^{d''-3}$, which has no neighbour in $[0]^3\overline{[1,1][0,0]^{\frac{d'-2}{2}}}$ -- a contradiction. 
\end{itemize}
\item The only couples of sites in $J_3$ at distance (at most) $2$ are of the form $([1,0,0]x,[0,1,0]x)$ for $x\in\overline{[1,1][0,0]^{\frac{d'-2}{2}}}$. Indeed, if the first two coordinates differ, the distance is at most $2$ only if all other coordinates are identical and if they do not differ, sites in $J_3$ are at distance $4$.
\item Every site $i_1\in I_0$ has a unique other site $i_2\in I_0$ at distance (at most) 2. Indeed, consider $d(i_1,i_2)\leq 2$ and argue that $i_1$ and $i_2$ only differ in coordinates $4$-$9$. If the last $d''$ coordinates differ, by less than 3, as $S$ is a $3$-snake, the time in the snake has different remainder modulo $3$ for the two sites and thus, $2$ of the the first $9$ coordinates must differ. If the last $d''$ coordinates differ by $3$ or more, $S$ being a $3$-snake implies $d(i_1,i_2)\geq 3$. Clearly, there is a unique site which differs from $i_1$ only in coordinates $4$-$9$.
\item $[0]^9S_t$ has common neighbours with $J_3$ only for $t=T-1$. Indeed, by condition~\ref{cond:cor3} of Lemma~\ref{cor:snakemodified} and parity one has $t\in\{T-1,T-3\}$. For $T-3$ it suffices to note that $[0]^9S_{T-3}$ has no two consecutive $1$s.
\item $[0]^9S_t$ has common neighbours with $J_2$ only for $t=T$ (More generally, $[0]^3a$ is at distance $2$ from $J_2$ only for $a=[0]^{d'}$).
\end{itemize}

\paragraph{Claim~1} At time $0\leq t<T$ the set of infected sites is $I\cup \{[0]^9S_{t'},t'\leq t\}$.
\begin{proof}[Proof of Claim~1]
We proceed by induction. 
\subparagraph{Base:} We show that $I\setminus \{[0]^9S_0\}$ is stable i.e. no uninfected site has three infected neighbours. Consider an uninfected site $s$ and split the reasoning in cases depending on $s$.
\begin{itemize}
\item If $s\in [0,0,1][*]^{d'}$, then it has at most two neighbours in $J_2$ (since $|J_2|=2$) and at most one neighbour in $I_0$. However, $J_2$ and $I_0$ have no common neighbours, so it has at most $2$ infected neighbours.
\item If $s\in \overline{[0][1]}[1][*]^{d'}\setminus J_2$, then it has one neighbour in $J_1$ (since this is a subcube), at most one neighbour in $J_2$, at most one neighbour in $J_3$ and no neighbours in $I_0$. However, $J_2$ and $J_3$ have no common neighbours, so it has at most $2$ infected neighbours.
\item If $s\in \overline{[0][1]}[0][*]^{d'}\setminus J_3$, then it has one neighbour in $J_1$ (since this is a subcube), at most one neighbour in each of $I_0$, $J_2$ and $J_3$. Indeed, for $J_3$ we know that all sites with (at least) two neighbours in $J_3$ are not in $\overline{[0][1]}[0][*]^{d'}$. However, $I_0$, $J_2$ and $J_3$ have pairwise no common neighbours, so we are done.
\item If $s\in[0]^3[*]^{d'}\setminus I_0$, then it has no neighbours in $J_1$ or $J_2$ and at most two in $J_3$, but since $J_3$ and $I_0$ have no common neighbours, it suffices to prove that $s$ cannot have $3$ neighbours in $I_0$. However, we know that each site in $I_0$ has common neighbours with only one other site in $I_0$, which concludes the proof of the base.
\end{itemize}

\subparagraph{Step:} Assume that at time $0\leq t<T-1$ the infected sites are $I\cup \{[0]^9S_{t'},t'\leq t\}$. We only need to check that none of the uninfected neighbours of $[0]^9S_t$ other than $[0]^9S_{t+1}$ has 3 infected neighbours at time $t$. 

As we know, $J_2$ and $J_3$ have no common neighbours with $[0]^9S_t$, so they cannot contribute. In $[0]^9[*]^{d''}$ the only infected site with neighbours in common with $[0]^9S_t$ is $[0]^9S_{t-2}$ (or none if $t\leq 1$), as $S$ is a $3$-snake. But their common neighbour different from $[0]^9S_{t-1}$ has no other infected neighbours in $[0]^9[*]^{d''}$ (since $S$ is a $3$-snake), does not neighbour $J_1$ (since it is in $[0]^9[*]^{d''}$), and nor does it neighbour $I_0$ (since the only neighbours of $I_0$ in $[0]^9[*]^{d''}$ are in $[0]^9S$ by construction).

Furthermore, the only other infected sites in $[0]^3[*]^{d'}\setminus[0]^9[*]^{d''}$ with common neighbours with $[0]^9S_t$ are the $4$ neighbours of $[0]^9S_{t\pm 1}$ in $I_0$. Recall that each of those has a common neighbour only with one other, so the only sites with three neighbours among those four and $[0]^9S_t$ are $[0]^9S_{t\pm 1}$. Moreover, $J_1$ does not contribute, as before, because those $4$ sites are in $[0]^3[*]^{d'}$.

Finally, the only infected site outside $[0]^3[*]^{d'}$ with common neighbours with $[0]^9S_t$ is $[1,1][0]^7S_t\in J_1$. Those common neighbours being outside $J_1$ and $[0]^3[*]^{d'}$, they cannot have more than $2$ infected neighbours in those two subcubes, which exhausts all possible cases and completes the induction step.
\end{proof}
\paragraph{Claim~2} The set $J_1\cup J_2\cup J_3\cup \{[0]^9S_{T-1}\}$ percolates.
\begin{proof}[Proof of Claim~2]
We have the following infections (we do not claim that they happen at different times or in this order).
\begin{itemize}
\item $\overline{[0][1]}[0]^7[1][0]^{d''-1}$ is infected by $[0]^9S_{T-1}$, $J_1$ and $J_3$.
\item $\overline{[0][1]}[0]^{d'+1}$ is infected by the previous one, $J_1$ and $J_2$.
\item $\overline{[0][1]}[0]\overline{[1][0]^{d'-1}}$ is infected by the previous one, $J_1$ and $J_3$.
\item $\overline{[0][1]}[0][*]^{d'}$ is infected by the previous one and $J_1$. Indeed, for all $2\le k\le d'$ every site in $\overline{[0][1]}[0]\overline{[1]^k[0]^{d'-k}}$ has at least $2$ neighbours in $\overline{[0][1]}[0]\overline{[1]^{k-1}[0]^{d'-k+1}}$ and one neighbour in $J_1$, so those sets become infected successively by induction.
\item $[0]^3[*]^{d'}$ is infected by the previous one and $[0]^9S_{T-1}$. Indeed, all sites in $[0]^3[*]^{d'}$ have two infected neighbours from the previous step, so they only need one more in order to be infected. But since $[0]^3[*]^{d'}$ is connected and contains the infected site $[0]^9S_{T-1}$, it does become infected entirely.
\item $\overline{[0][1]}[1][*]^{d'}$ is infected by $\overline{[0][1]}[0][*]^{d'}$, $J_1$ and $J_2$ just like in the previous step.
\item $[0,0][1][*]^{d'}$ is infected by the ones in the previous two steps.
\end{itemize}
Hence, the whole hypercube is infected.
\end{proof}
The lemma follows trivially from the two claims.
\end{proof}

The next lemma establishes our upper bound on the percolation time.
\begin{lem}
\label{lem:uppermain}
Let $r\geq 3$. Then for all $d\geq r$
\[M_r(d)\leq (4r+2)\frac{2^d}{d}\,.\]
\end{lem}
\begin{proof}
Assume that $M_r(d)> (4r+2)\frac{2^d}{d}$ for some $d$ and consider a percolating set of initially infected sites, which achieves the maximal time. For each site $v$ of the hypercube denote $t_v$ its percolation time. Note that any site $v$ has at most $r-1$ neighbours $u$ such that $t_v-t_u>1$, so there are at most $(r-1)2^d$ edges $uv$ of the hypercube such that $|t_v-t_u|\geq 2$. Call a site $v$ \emph{bad} if it has at least $\frac{d}{2}$ neighbours $u$ such that $|t_v-t_u|\geq 2$ and \emph{good} otherwise. Thus, there are at most $\frac{2^{d+2}(r-1)}{d}$ bad sites in total, since an edge contributes to at most $2$ of them. But then there are at most $\frac{2^{d+2}(r-1)}{d}$ values of $t$ when a bad site becomes infected and in particular there are more than $\frac{2^{d}}{d}((4r+2)-4(r-1))=6\frac{2^d}{d}$ values when a good site becomes infected. But if $v$ is a good site, then at least $\frac{d}{2}$ of its neighbours are infected at time $t_v-1$, $t_v$ or $t_v+1$. Hence, applying this to one good site for each time when there is one, one obtains that there are more than
\[6\cdot \frac{2^{d}}{d}\cdot\frac{d}{2}\cdot\frac{1}{3}=2^d\]
infected sites, since each one is counted up to three times -- a contradiction.
\end{proof}
\begin{rem}
In order to obtain a better constant with the same proof, bad sites should be defined to have $Cd$ edges of the type specified and $C$ should then be optimised.
\end{rem}
The main result now follows immediately.
\begin{proof}[Proof of Theorem~\ref{th:main}]
Let us first prove the lower bound. For $r>3$, consider a configuration giving $M_3(d-r+3)$ in a $(d-r+3)$-dimensional subcube and infect the rest of the hypercube. Then sites in that subcube follow exactly the 3-neighbour bootstrap process restricted to it and thus the problem is reduced to $r=3$. For $r=3$ the result follows directly from Lemma~\ref{lem:time} and Proposition~\ref{prop:Evdokimov}, so we are done when $d$ is odd. Consider $d\geq 15$ even and denote by $A$ a $d-1$ dimensional percolating set achieving $M_3(d-1)$. Then we claim that the $d$-dimensional set $A':=[*]A$ percolates in exactly the same time. Indeed, by an immediate induction at any time $t$ a site $[0]a\in [*]^{d}$ is infected if and only if $[1]a$ is, so for any uninfected site $[0]b$ the only infected neighbours are in $[0][*]^{d-1}$ and so, by induction it becomes infected if and only if $b$ becomes infected at time $t$ in the $d-1$-dimensional process.

The upper bound was proved in Lemma~\ref{lem:uppermain}.
\end{proof}

\section{Conclusion and open problems}
In conclusion, our result exhibits a significant difference between the $2$ and $3$-neighbour models on the hypercube. The reason why our method does not work for the $2$-neighbour case is that parasite infections are inevitable. More precisely, there necessarily appear additional infections around an infected path -- at each `corner' of the path (in the hypercube a path has `corners' at each step) at the first time step and more afterwards.

Further understanding of the different behaviours should be of use in attacking the $3$-neighbour model on the hypercube with random initial condition, by showing what anomalies one needs to take into consideration. We list here a few of the questions raised by the present work, not necessarily aiming directly at solving that model.

The first natural question to ask in view of our work is to determine the exact order of the maximal percolation time. We conjecture that the upper bound is tight up to a constant.
\begin{conj}
Prove that for all fixed $r\geq 3$
\[M_r(d)=\Theta\left(\frac{2^d}{d}\right)\,.\]
\end{conj}
It should be noted, that this result would follow from the same proof, if one establishes the corresponding lower bound for the maximal length of $3$-snakes, improving the result of~\cite{Evdokimov76}.

Secondly, a probably difficult question is to determine the random percolation time. The probabilistic counterpart of our extremal result would be as follows.
\begin{question}
Conditionally on percolating, what is the order of the percolation time if the initially infected sites are chosen randomly and independently with probability $p(d)$?
\end{question}
It would, namely, be interesting to see if exponentially large times such as the ones we give manage to alter the mean percolation time despite their low probability of occurrence.

Finally, in view of the more recent development of $\mathcal{U}$-bootstrap percolation in $2$ dimensions~\cite{Bollobas14,Bollobas15} and, currently in higher, but fixed number of dimensions, one could ask for similar results about models more general than the $r$-neighbour model, but still on the hypercube (e.g. a site is infected if some fixed subset, defined up to isomorphism, of its $2$-neighbourhood is infected). An answer of satisfactory generality to the following question might need to wait until $\mathcal{U}$-bootstrap percolation setting is extended to the hypercube, but it is worth investigating nonetheless.
\begin{question}
When is the order of the maximal time of $\mathcal{U}$-bootstrap percolation on the hypercube up to a constant given by the maximal length of a $k$-snake for some $k$ and how is $k$ determined by $\mathcal{U}$?
\end{question}

\section*{Acknowledgements}
The author would like to thank Stephen G. Z. Smith -- for numerous particularly helpful discussions and encouragement, Dmitry Chelkak -- for crucial bibliographic help and Lyuben Lichev -- for meticulous proofreading.

\bibliographystyle{plain}
\bibliography{D:/Master/Study/LaTeX/Bib/Bib}
\end{document}